\documentclass{amsart}

\usepackage{amsthm, amsfonts, amssymb, amsmath, latexsym, enumerate, times}
\usepackage[latin1]{inputenc}
\usepackage{mathrsfs}
\usepackage{array}
\usepackage{comment}
\usepackage{paralist}
\usepackage{color}

\newtheorem{theorem}{Theorem}

\theoremstyle{definition}

\newtheorem{remark}[theorem]{Remark}

\def\geq{\geqslant}
\def\leq{\leqslant}

\begin{document}
 
\title[On the minimal genus]{Remarks on the minimal genus of curves linearly moving on a surface}

\author{Ciro Ciliberto}
\address{Dipartimento di Matematica, Universit\`a di Roma Tor Vergata, Via O. Raimondo
 00173 Roma, Italia}
\email{cilibert@axp.mat.uniroma2.it}
 
 
 
\maketitle

\begin{abstract} Given a smooth, irreducible, projective surface $S$, let $g(S)$ be the minimum geometric genus of an irreducible curve that moves in a linear system of positive dimension on $S$. We determine the value of this birational invariant for a general surface of degree $d$ in $\mathbb P^3$ and give a bound for $g(S)$ if $S$ is a general polarised K3 or abelian surface.   
\end{abstract}

\section*{Introduction} Let $S$ be a smooth, irreducible, projective surface $S$. One can introduce a basic birational invariant of $S$ to be the minimum geometric genus of an irreducible curve on $S$. Determining this invariant is not difficult in some cases. For instance for surfaces with negative Kodaira dimension and for Enriques surfaces it is 0. It is 1 for bielliptic surfaces and 2 for general abelian surfaces. It is at most 1 for elliptic surfaces. However in general, determining this invariant is quite complicated. This question
has been examined in a number of interesting cases. For instance if $S$ is a very general surface of degree $d\geq 5$ in $\mathbb P^3$ (see \cite {Xu}) and if $S$ is a K3 surface (in this case this invariant is 0, see \cite {MM}, Appendix).   

In this paper we want to introduce a related but different birational invariant $g(S)$ of a surface $S$. Namely $g(S)$ is by definition the minimum geometric genus of an irreducible curve that moves in a linear system of positive dimension on $S$. Again there are cases in which it is easy to determine $g(S)$. For instance $g(S)=0$ characterizes rational surfaces. For ruled surfaces $S$ of genus $g$ one has $g(S)=g$. The value $g(S)=1$ characterizes non--rational surfaces with a rational pencil of elliptic curves, like, for instance, Enriques surfaces or regular elliptic surfaces. However the determination of $g(S)$ is quite complicated in general. 

In this note we determine $g(S)$ for $S$ a very general surface of degree $d\geq 5$ in $\mathbb P^3$, proving that it equals the genus of plane sections of $S$ and actually only positive dimensional linear system of plane sections have curves with that genus (see Theorem \ref {thm:main1} below). Then we give a lower bound for $g(S)$ for general polarised K3 and abelian surfaces $(S,H)$ (see Theorems \ref {thm:main2} and  \ref {thm:main3}), though we do not believe that this bound is in general sharp.    

The idea of the proof of the above three results is the same for all of them and it is quite elementary: it consists in  bounding the genus of curves moving in rational pencils by using basic computations of genus and self--intersection of such curves.

\medskip

{\bf Acknowledgements:} The author is a member of GNSAGA of INdAM. 

As soon as this note appeared on the math arxiv, David Stapleton kindly pointed out to the author the paper \cite {EL}, in which in turn the paper \cite {Ko} was cited. Unfortunately the author was not aware of these two papers which contain results that strictly include the ones of this note. The author is very grateful to David Stapleton.  

\section{Surfaces in $\mathbb P^3$}\label{sec:p3} 

Let $S$ be a smooth surface of degree $d\geq 5$ in $\mathbb P^3$ (in the cases $d\leq 3$ the surface $S$ is rational hence $g(S)=0$, and in the case $d =4$ the surface $S$ is a K3 surface and we will deal with it in the sext section). We denote by $H$ the plane section class of $S$. We will assume that ${\rm Pic}(S)\cong \mathbb Z$ generated by $H$, a property that is enjoyed by very general surfaces of degree $d\geq 4$ by the  Noether--Lefschetz theorem. 

If $m$ is any positive integer, we denote by $g_{d,m}$ the arithmetic genus of the curves in the linear system $|mH|$. One has
\begin{equation*}\label{eq:gen}
g_{d,m}=\frac {md(m+d-4)}2 +1.
\end{equation*}
We set $\dim(|mH|)=r_{d,m}$ and one has
\[
r_{d,m}= \Bigg \{  \begin{array}{ccc}
&\frac {
m^ 3+6m^ 2+11m+6}6 \,\,\, &\text {if}\,\,\, m<d \\
& \frac {d\big ( 3m^ 2+3m (4-d)+(d^ 2-6d+11)\big)}6
 \,\,\, &\text {if}\,\,\, m\geqslant d. \\
\end{array}
\]

\begin{theorem}\label{thm:main1} Let $S$ be a  surface of degree $d\geq 5$ in $\mathbb P^3$ such that ${\rm Pic}(S)\cong \mathbb Z$ generated by $H$. Then 
$$g(S)=g_{d,1}=\frac {d(d-3)}2 +1$$
and if $\mathcal L\subset |mH|$ is a linear system of positive dimension whose general curve is irreducible of genus $g(S)=g_{d,1}$, then $m=1$ and   $\mathcal L$ has at most simple base points.
\end{theorem}

\begin{proof} Of course $g(S)\leq g_{d,1}$. We want to prove that  equality holds.

Let $\mathcal L\subseteq |mH|$ be a linear system of positive dimension $r$ whose general curve is irreducible of genus $g$. By imposing $r-1$ general simple base points to $\mathcal L$ we can assume that $\mathcal L$ is a pencil, i.e., $r=1$. Then let $p_1,\ldots, p_h$ be the base points of $\mathcal L$ (they can be proper or infinitely near), and let $m_1,\ldots, m_h$ be their respective multiplicities. We have
\begin{equation}\label{eq:deg}
m^2d=(mH)^2=\sum_{i=1}^h m_i^2.
\end{equation}
By adjunction and by \eqref {eq:deg}, we have
\begin{equation}\label{eq:gen2}
2g-2=md(d-4)+\sum_{i=1}^h m_i.
\end{equation}

Moreover we have
$$
1=\dim(\mathcal L)\geq r_{d,m}- \sum_{i=1}^h \frac {m_i(m_1+1)}2
$$
i.e., 
$$
\sum_{i=1}^h m_i^2+\sum_{i=1}^hm_i-2r_{d,m}+2\geq 0,
$$
and therefore
$$
\Big(\sum_{i=1}^h m_i\Big)^2+\sum_{i=1}^hm_i-2r_{d,m}+2\geq 0,
$$
whence
$$
\sum_{i=1}^hm_i\geq \frac {-1+\sqrt{1+8(r_{d,m}-1)}} 2.
$$
By \eqref {eq:gen2}, we get
$$
g\geq \frac {3+\sqrt{1+8(r_{d,m}-1)}} 4+ \frac {md(d-4)}2.
$$
If $m\geq 2$ then $r_{d,m}\geq 9$, so we get 
$$
g\geq \frac {3+\sqrt {65}}4+d(d-4)=\frac {3+\sqrt {65}}4+2g_{d,1}-2-d.
$$
Since $d\geq 5$, we have $g_{d,1}>d$, hence 
$$
g>\frac {\sqrt {65}-5}4+g_{d,1}
$$
so $g>g_{d,1}$. This implies that $g=g(S)$ only if $m=1$, as wanted. \end{proof}

\section{K3 and abelian surfaces}\label{sec:k3}

In this section we let $(S,H)$ be a polarised K3 surface, with $H^2=2p-2$ and ${\rm Pic}(S)\cong \mathbb Z$ generated by $H$. 

If $m$ is any positive integer, we denote by $\gamma_{p,m}$ the arithmetic genus of the curves in the linear system $|mH|$. One has
\begin{equation}\label{eq:genk3}
\gamma_{p,m}=m^2(p-1)+1
\end{equation}
and 
$$
\dim(|mH|)=\gamma_{p,m}. 
$$

\begin{theorem}\label{thm:main2} Let $(S,H)$ be a polarised K3 surface, with $H^2=2p-2$ and ${\rm Pic}(S)\cong \mathbb Z$ generated by $H$.  Then 
$$
g(S)\geq \frac {3+\sqrt {1+8(p-1)}} 4.
$$
\end{theorem}

\begin{proof} Let $\mathcal L\subseteq |mH|$ be a linear system of positive dimension $r$ whose general curve is irreducible of genus $g(S)$. As in the proof of Theorem \ref {thm:main1}, we can assume that $\mathcal L$ is a pencil.  Let $p_1,\ldots, p_h$ be the base points of $\mathcal L$, and let $m_1,\ldots, m_h$ be their respective multiplicities. We have
$$
2m^2(p-1)=(mH)^2=\sum_{i=1}^h m_i^2.
$$
from which we deduce
\begin{equation}\label{eq:gen3}
2g(S)-2=\sum_{i=1}^h m_i.
\end{equation}

Moreover we have
$$
1=\dim(\mathcal L)\geq \dim(|mH|)- \sum_{i=1}^h \frac {m_i(m_1+1)}2=m^2(p-1)+1-\sum_{i=1}^h \frac {m_i(m_1+1)}2
$$
thus
$$
\sum_{i=1}^h m_i^2+\sum_{i=1}^hm_i-2m^2(p-1)\geq 0,
$$
and therefore
$$
\Big(\sum_{i=1}^h m_i\Big)^2+\sum_{i=1}^hm_i-2m^2(p-1)\geq 0,
$$
whence
$$
\sum_{i=1}^hm_i\geq \frac {-1+\sqrt{1+8m^2(p-1)}} 2.
$$
The assertion follows by \eqref {eq:gen3}. \end{proof}

\begin{remark}\label{rem:k3} We do not believe the bound for $g(S)$ in Theorem \ref {thm:main2} is sharp in general. However it is so for $p=2$. Indeed, in this case,  Theorem \ref {thm:main2} gives $g(S)\geq 2$. On the other hand, of course $g(S)\leq p=2$, hence $g(S)=2$. Moreover, from the proof of Theorem \ref {thm:main2}, we see that if $\mathcal L\subseteq |mH|$, with $m\geq 2$, has irreducible general curves of genus $g$ then $g>1$. So we can have $g=g(S)$ only if $m=1$. 

In the case $p=3$, the proof of Theorem \ref {thm:main2} implies the result $g(S)=3$ which is attained for $m=1$ and possibly only for $m=2$. However we believe that  linear systems $\mathcal L\subseteq |2H|$ of positive dimension with general irreducible curve $C$ are such that the geometric genus of $C$ is larger than $3$, so that only the case $m=1$ should be possible.
\end{remark}

The proof of Theorem \ref {thm:main2} can be easily adapted to the case of abelian surfaces. We limit ourselves to stating the result. 

\begin{theorem}\label{thm:main3} Let $(S,H)$ be a polarised abelian surface, with 
a polarization of type $(1,a)$ and ${\rm Num}(S)\cong \mathbb Z$ generated by $H$.  Then if $a\geq 2$ one has 
\begin{equation}\label{eq:lops}
g(S)\geq \frac {3+\sqrt {8a-15}} 4,
\end{equation}
whereas if $a=1$ one has  $g(S)\geq 3$. 
\end{theorem}

Again we believe the bound for $g(S)$ in Theorem \ref {thm:main3} is far from being sharp.  In particular, the bound  \eqref {eq:lops} is non--trivial only for $a$ large enough. Indeed for $2\leq a\leq 4$ it gives $g(S)\geq 1$ and for $a=5$ it gives $g(S)\geq 2$, that are trivial bounds.


\begin{thebibliography}{}

\bibitem {EL} L. Ein, R. Lazarsfeld, \emph{The Konno invariant of some algebraic varieties}, European J. of Math., {\bf 6} (2), (2020), 420--429.

\bibitem {Ko} K. Konno, \emph{Minimal pencils on smooth surfaces in $\mathbb P^3$}, Osaka J. of Math., {\bf 45} (3), (2008--09), 789--805.

\bibitem {MM} S. Mori, S. Mukai, \emph{The uniruledness of the moduli space of curves of genus 11}, Algebraic Geometry,
 Proc. Tokyo/Kyoto, 334--353, Lecture Notes in Math. {\bf 1016}, Springer, Berlin, 1983.

\bibitem {Xu} G. Xu, \emph{Subvarieties of general hypersurfaces in projective space},  Journal of Differential Geometry {\bf 39}, no. 1 (1994), 139--172.




 
\end{thebibliography}
\end{document}